\documentclass[11pt]{amsart}

\usepackage{amsmath,amsfonts,amsthm,graphicx, color}
\usepackage{amsmath}
\usepackage{graphicx,psfrag}  
\usepackage[psamsfonts]{amssymb}
\usepackage{amscd}
\usepackage[all,cmtip]{xy}
\usepackage{hyperref}

\title[Symplectomorphisms with positive metric entropy]{Symplectomorphisms with positive metric entropy%\\(Preliminary version)
}
\author[A. Avila, S. Crovisier, A. Wilkinson]{A. Avila, S. Crovisier and A. Wilkinson}
\date{\today}

\thanks{
S.C. was partially supported by the ERC project 692925 NUHGD.
A.W. was supported by NSF grant DMS$-1402852$.
}

\address{Artur Avila \newline
\rm Institut f\"ur Mathematik,
Universit\"at Z\"urich\newline
Winterthurerstrasse 190,
CH-8057 Z\"urich,
Switzerland
\newline \& IMPA, Estrada Dona Castorina 110, Rio de Janeiro, Brazil.}
\address{Sylvain Crovisier \newline
\rm CNRS - Laboratoire de Math{\'e}matiques d'Orsay, UMR 8628\newline
Universit{\'e} Paris-Sud 11, 91405 Orsay Cedex, France.}
\address{Amie Wilkinson \newline
\rm Department of Mathematics,
University of Chicago, \newline 5734 S. University Avenue Chicago, Illinois 60637, USA.}
\theoremstyle{plain}
\newtheorem{theorem}{Theorem}
\newtheorem{lemma}{Lemma}[section]

\newtheorem{proposition}[lemma]{Proposition}

\newtheorem*{corollary*}{Corollary}
\newtheorem*{claim}{Claim}

\newtheoremstyle{vThm*}%
{}{}%
{\itshape}%
{-3pt}{\bfseries}%
{}{ }%
{\thmnote{#3}}%
\theoremstyle{vThm*}
\newtheorem*{nThm*}{}

\theoremstyle{definition}

\newtheorem*{definition*}{Definition}
\newtheorem{remark}[theorem]{Remark}
\newtheorem{example}[theorem]{Example}

\def\endproof{$\diamond$ \bigskip}

\def\Diff{\operatorname{Diff} }
\def\Symp{\operatorname{Symp} }

\def\title{\em}

\def\bar{\overline}

\def\id{\hbox{id}}

\def\cW{\mathcal{W}}

\def\cB{\mathcal{B}}
\def\cD{\mathcal{D}}
\def\cF{\mathcal{F}}

\def\cZ{\mathcal{Z}}
\def\cF{\mathcal{F}}
\def\cU{\mathcal{U}}
\def\cL{\mathcal{L}}

\def\cG{\mathcal{G}}
\def\cR{\mathcal{R}}

\def\C{\mathcal{C}}
\def\cC{\mathcal{C}}

\def\cO{\mathcal{O}}

\def\P{\mathcal{P}}

\def\Z{D}

\def\cP{\mathcal{P}}

\def\transverse{\,\raise2pt\hbox to1em{\hfil$\top$\hfil}\hskip -1em \hbox
to1em{\hfil$\cap$\hfil}\,} 

\newcommand\RR{{\mathbb R}}
\newcommand\CC{{\mathbb C}}
\newcommand\PP{{\mathbb P}}

\newcommand\ZZ{{\mathbb Z}}

\newlength{\figboxwidth} 

\def\W{\mathcal{W}}

\begin{document}

\begin{abstract}
We  obtain a dichotomy for $C^1$-generic symplectomorphisms:
either all the Lyapunov exponents of almost every point vanish, or the map is partially hyperbolic and ergodic with respect to volume. This completes a program first put forth by Ricardo Ma\~n\'e.

A main ingredient in our proof is a generalization to partially hyperbolic invariant sets of the main result in \cite{DW} that stable accessibility is $C^1$ dense among partially hyperbolic diffeomorphisms.
\end{abstract}

\maketitle
%\addcontentsline{toc}{section}{\mbox{}\quad\quad Introduction}
%\DeactivateToc
\section*{Introduction}
A measurable map $f\colon M\to M$ is {\em ergodic} with respect to an invariant probability measure $\mu$ if every $f$-invariant subset of $M$ is $\mu$-trivial: $f^{-1}(A) = A$ implies $\mu(A) = 0$ or $1$, for every measurable $A\subset M$.   In the context of this paper, where $M$ is a closed manifold, $f$ is a homeomorphism, and $\mu=m$ is a normalized volume, ergodicity is equivalent to the equidistribution of almost every orbit: for $m$-almost every $x\in M$ and every continuous $\phi\colon M\to \RR$,
\[
\lim_{n\to \infty} \frac{1}{n} \sum_{j=1}^{n} \phi(f^j(x)) = \int_{M}\phi\, d{m}.
\]

   In his 1983 ICM address~\cite{M},
Ma\~n\'e announced the following result, whose proof was later completed by Bochi~\cite{B}.

\begin{theorem}[Ma\~n\'e-Bochi]
$C^1$-generically, an  area preserving diffeomorphism $f$ of a closed,
connected surface $M^2$ is either Anosov and ergodic or satisfies
\[\lim_{n\to\pm \infty} \frac{1}{n} \log\| D_x f^nv\|
= 0,
\]
for a.e. $x\in M$ and every  $0 \neq v \in T_xM$.
\end{theorem}

In \cite{ACW},  we proved the optimal generalization of this result to volume-preserving diffeomorphisms in any dimension:

\begin{theorem}[\cite{ACW}]\label{t=ACW}
$C^1$-generically, a volume-preserving diffeomorphism $f$ of a closed, connected
manifold $M$ is either nonuniformly Anosov and ergodic or satisfies
\begin{equation}\label{e=allvanish}\lim_{n\to\pm \infty} \frac{1}{n} \log\| D_x f^nv\|
= 0,
\end{equation}
for a.e. $x\in M$ and every  $0 \neq v \in T_xM$.
\end{theorem}

The ``nonuniformly Anosov" condition in Theorem~\ref{t=ACW} implies in particular that there is a constant $c>0$ such that for almost every  $x\in M$ and every  $0 \neq v \in T_xM$, either
\[\hbox{$\lim_{n\to\pm \infty} \frac{1}{n} \log\| D_x f^nv\| > c$, or
$\lim_{n\to\pm \infty} \frac{1}{n} \log\| D_x f^nv\| <-c$}.\]  
The nonuniformity in this conclusion cannot be removed:  in dimension greater than $2$, there are $C^1$-open sets of volume-preserving diffeomorphisms with positive entropy that are not Anosov.  These include, but are not limited to, the partially hyperbolic diffeomorphisms (we define the Anosov condition and partial hyperbolicity below).

Theorem~\ref{t=ACW} can be rephrased using Ruelle's inequality, which implies that for a volume-preserving diffeomorphism, equation (\ref{e=allvanish}) holds for  almost every $x\in M$ and every nonzero $v\in T_xM$ if and only if the volume entropy $h_{m}(f)$ vanishes.  Thus Theorem~\ref{t=ACW} implies that {\em  $C^1$-generically among volume-preserving diffeomorphisms, positive volume entropy (i.e.~$h_{m}(f)>0$) implies ergodicity.}

\subsection*{Ergodicity of symplectomorphisms}
The focus of Ma\~n\'e's discussion in \cite{M} was in fact the $C^1$-generic behavior of {\em symplectomorphisms}, which in dimension $2$ coincide with the area-preserving diffeomorphisms.  If $f\colon M^{2n}\to M^{2n}$ preserves a symplectic form $\omega$, then it preserves the normalized volume $m$ induced by the  form $\omega^n$.  The question of whether $f$ is typically ergodic with respect to this volume has a long history going back to the ergodic hypothesis for Hamiltonian systems.  

For symplectomorphisms, the exact conclusion of Theorem~\ref{t=ACW} does not hold; in particular, $C^1$-generically among the {\em partially hyperbolic} symplectomorphisms that are not Anosov, one has positive entropy {\em without the nonuniformly Anosov condition}.  On the other hand, the $C^1$-generic partially hyperbolic symplectomorphism is ergodic \cite{ABW}.  This leaves the natural question: {\em for the  $C^1$-generic symplectomorphism, does positive volume entropy imply partial hyperbolicity, and hence ergodicity?}  
 
In the same address \cite{M}, Ma\~n\'e announced that for the $C^1$-generic symplectomorphism, positive volume entropy implies the existence of a partially hyperbolic invariant set of positive volume.  A proof of this claim, requiring substantially new techniques, was provided nearly 20 years later by Bochi \cite{Bsymp}.    In this paper, we take the Bochi result as a starting point to prove the full generalization of the Ma\~n\'e-Bochi theorem to symplectomorphisms: 
 
\begin{nThm*}{{\bf Theorem A.}} $C^1$ generically among the symplectomorphisms of a  compact, connected symplectic manifold $(M,\omega)$, positive volume entropy implies partial hyperbolicity and ergodicity.
\end{nThm*}

Note that there are obstructions to partial hyperbolicity on certain symplectic manifolds (see \cite{Leonidstudent} for a discussion); for example $\CC \PP^n$ does not carry a partially hyperbolic symplectomorphism.  For these manifolds, Theorem A implies that the $C^1$ generic symplectomorphism has volume entropy 0.  We also remark that the assumption ``positive volume entropy" cannot be replaced by ``positive topological entropy:" on any symplectic manifold there exist symplectic horseshoes with positive topological entropy.  These horseshoes persist  under $C^1$-small perturbation.

Needless to say, the techniques behind the proof of Theorem A are essentially disjoint from those in the volume-preserving setting of Theorem~\ref{t=ACW}.
In the volume-preserving setting, the positive entropy condition implies the existence of nonzero Lyapunov exponents on the phase space, and the proof in \cite{ACW} harnesses the presence of some nonzero exponents to eliminate all zero Lyapunov exponents
throughout large parts of the phase space.  A Baire argument completes the proof. In the symplectic setting, we prove that $C^1$ generically, the partially hyperbolic set provided by \cite{Bsymp} in the presence of positive entropy is the entire manifold.  The main result in \cite{ABW} then gives the conclusion.  We now explain this argument in more detail.

\subsection*{Partial hyperbolicity and accessibility}

Let $f\colon M\to M$ be a diffeomorphism.  A compact, $f-$invariant set $\Lambda\subseteq M$
is {\em partially hyperbolic} if there exists $N\geq 1$ and a $Df$-invariant splitting of the tangent bundle over $\Lambda$:
\begin{equation} \label{e=splitting} T_\Lambda M = E^u \oplus E^c\oplus E^s,
\end{equation}
such that for every $x\in \Lambda$ and all unit vectors $v^u\in  E^u_x, v^c\in E^c_x$, and $v^s\in E^s_x$, we have
\[ \|D_xf^{N}(v^{s})\| \leq \frac12 \|D_xf^{N}(v^{c})\| \leq  \frac14 \|D_xf^{N}(v^{u})\|,
\]
and
\[
\max\{\|D_xf^{N}(v^{s})\|,  \|D_xf^{-N}(v^{u})\|\} < \frac12.
\]
We assume throughout that the bundles $E^u$ and $E^s$ in the splitting (\ref{e=splitting}) are nontrivial.  This partially hyperbolic splitting is always continuous.

A diffeomorphism of a closed manifold $M$ is {\em partially hyperbolic} if $M$ is a partially hyperbolic set for $f$, and {\em Anosov} if it is partially hyperbolic, with $E^c=\{0\}$.

Let $\Lambda$ be a compact partially hyperbolic set for $f$.  Through each $x\in \Lambda$ are unique local stable and unstable manifolds $\cW^s_f(x,\hbox{loc})$ and $\cW^u_f(x,\hbox{loc})$, respectively, which are given by a graph transform argument in a suitable neighborhood of $\Lambda$.   The local stable and unstable manifolds determine global manifolds by
\[\cW^u_f(x) = \bigcup_{n\geq 0} f^{n}(\cW^u_f(f^{-n}(x),\hbox{loc}),\quad\hbox{and } \cW^s_f(x) = \bigcup_{n\geq 0} f^{-n}(\cW^s_f(f^{n}(x),\hbox{loc}).
\]

We say that $\Lambda$ is {\em $u$-saturated} if for any $x\in \Lambda$, $\cW^u_f(x)\subset \Lambda$ and  {\em$s$-saturated} if  for any $x\in \Lambda$, $\cW^s_f(x)\subset \Lambda$.  We say that $\Lambda$ is {\em bisaturated} if it is both $s$- and  $u$-saturated.
The bisaturated set $\Lambda$ is {\em accessible} if for every $p,q \in \Lambda$ there is an {\em $su$-path}  for
$f$  in $\Lambda$ -- that is, a piecewise $C^1$ path such that every segment is contained in a single leaf of $\W^s_f$ or a single leaf of $\W^u_f$ -- from $p$ to $q$.

Note that if $f$ is partially hyperbolic, then $M$ is automatically bisaturated.
In this case $f$ is accessible if for every $p,q \in M$ there is an {\em $su$-path} from $p$ to $q$.
Dolgopyat and Wilkinson proved in \cite{DW} that accessibility holds for a $C^1$ open and dense set of partially hyperbolic diffeomorphisms, volume-preserving diffeomorphisms, and symplectomorphisms of a closed, connected manifold $M$.

Our main result is a local version of the main result in~\cite{DW} that implies both Theorem A and the results of \cite {DW}.\footnote{Our proof of  Theorem B also corrects some omissions in the proof in \cite{DW}.
We will indicate where.}
Let us denote by $\Diff^k(M)$ the space of $C^k$ diffeomorphisms endowed with the $C^k$ topology.
If $m$ is a volume form on $M$, we denote by $\Diff^k_m(M)$ the subspace of $C^k$ diffeomorphisms that preserve $m$;
if $(M^{2n},\omega)$ is a symplectic manifold, we denote by $\Symp^k(M,\omega)$ the subspace of $C^k$ symplectomorphisms.

\begin{nThm*}{{\bf Theorem B.}}\label{t.perturbation}
Let $M$ be a closed manifold, let  $\Lambda$ be a partially hyperbolic set of a diffeomorphism $f\colon M\to M$, and let $\cU$ be a neighborhood of $f$ in $\Diff^1(M)$.
There exists a neighborhood $U$ of $\Lambda$ and a non-empty open set $\cO\subset \cU$ such that:
for any $g\in \cO$, any bisaturated partially hyperbolic set $\Delta \subset U$ for $g$ has non-empty interior and is accessible.

The same result holds in $\Diff^1_m(M)$ and in $\Symp^1(M,\omega)$, if $(M^{2n}, \omega)$ is a symplectic manifold.
\end{nThm*}

Using Theorem~B, we give a proof of Theorem~A.

\begin{proof}[Proof of Theorem A] Let $(M,\omega)$ be a closed symplectic manifold.  Bochi proved \cite{Bsymp} that there are two disjoint open sets, $\cZ$ and $\cP$ in $\Symp^1(M,\omega)$, such that
\begin{itemize}
\item[--] $\cZ\cup \cP$ is dense in $\Symp^1(M,\omega)$;
\item[--] for $f$ in a residual subset of $\cZ$, the volume entropy $h_m(f)$ is zero;
\item[--] for $f\in \cP$, there exists a positive volume, partially hyperbolic $f$- invariant set $\Delta_f$.
\end{itemize}
The openness of $\cP$ follows from\footnote{\cite{AB} is written in the volume-preserving case, but remains true restricted to any closed (or Baire) subspace of the space of volume preserving maps. In particular it holds for symplectomorphisms.} Theorem C in \cite{AB}, which also implies that
the set $\cR$ of continuity points of $f\mapsto \Delta_f$ is residual in $\cP$. Because the set $\Delta_f$ has positive volume, it is  bisaturated: see \cite[Corollary B]{Zhang}.

We  consider the diffeomorphisms in $\cP$.
Let $f\in \cR\subset \cP$.  Theorem B  implies that
there exists an open set $\cU\subset \cP$ containing $f$ in its closure such that
for $g\in \cU$, the set $\Delta_g$ has nonempty interior.
Theorem 1 in   \cite{ABC} implies that $g$ is also transitive when it belongs to a residual subset of $\cU$, implying that $\Delta_g=M$, and so
$M$ is partially hyperbolic.  Since partial hyperbolicity is robust,
we have thus shown that for $g$ in an open and dense subset of $\cP$, the whole manifold is a partially hyperbolic set, i.e. $g$ is partially hyperbolic.

Theorem A in \cite{ABW}  states that among the $C^1$, partially hyperbolic symplectomorphisms, ergodicity is $C^1$-generic, completing the proof.
\end{proof}

The  boundary of a bisaturated partially hyperbolic set is also bisaturated (see  Lemma~\ref{l=bisaturatedboundary} in Section~\ref{ss=bisaturated}).
When $M$ is connected, this has the following consequence, which generalizes the main result in \cite{DW}.

\begin{nThm*}{{\bf Corollary C.}}  Let $M$ be a closed, connected manifold, let  $\Lambda$ be a partially hyperbolic set of a diffeomorphism $f\colon M\to M$, and let $\cU$ be a neighborhood of $f$ in $\Diff^1(M)$.

There exists a neighborhood $U$ of $\Lambda$ and a non-empty open set $\cO\subset \cU$ such that for any $g\in \cO$, there is no proper bisaturated subset of $U$.

In particular, if $f$ is partially hyperbolic, then there exists a nonempty $\cO\subset \cU$ such that every $g\in \cO$ is accessible.
\end{nThm*}

\begin{proof} Let $\Lambda$ be a partially hyperbolic set for $f$, and let $\cU$ be a neighborhood of $f$ in $\Diff^1(M)$.  Applying Theorem B, we obtain a neighborhood $U$ of $\Lambda$ and a non-empty open set $\cO\subset \cU$ such that for $g\in \cO$, any bisaturated set $\Delta$ for $g$ in $U$ has empty interior and is accessible.   

Thus, if $U$ contained a proper bisaturated set for a diffeomorphism $g\in \cO$, then its boundary would be a bisaturated set with empty interior, a contradiction.

If $f$ is partially hyperbolic, then applying this argument to $\Lambda=M$ gives that any $g\in \cO$ is accessible.
\end{proof}

 We remark that in the dissipative setting an earlier version of Corollary C was proved  for bi-Lyapunov homoclinic classes by Abdenur-Bonatti-Diaz \cite{ABD}.  Corollary C has the following direct corollary.

\begin{nThm*}{{\bf Corollary D.}} Let $M$ be closed and connected.  Then the $C^1$-generic $f$ in $\Diff^1(M)$  has no proper, partially hyperbolic, bisaturated invariant compact set.
\end{nThm*}

\begin{proof}  Let $\cB$ be a countable basis for the topology on $M$ (not including $M$ itself). 
For $U$ in $\cB$, let
$\cC_U$ be the set of diffeomorphisms $f\in \Diff^1(M)$
whose maximal invariant set in $\overline U$ is partially hyperbolic and let
$\cD_U = \Diff^1(M)\setminus \overline{\cC_U}$.  Clearly $\cC_U\cup \cD_U$ is open and dense in $\Diff^1(M)$.

By Corollary C, there exists a dense open subset $\cG_U\subset \cC_U$ such that $U$ does not contain any proper bisaturated set.
Now let
\[\cR := \bigcap_{U\in \cB}\left(\cC_U \cup \cG_U\right).\]
The set $\cR$ is residual in $\Diff^1(M)$, and $g\in \cR$ implies that $g$ has no proper bisaturated partially hyperbolic subsets.
\end{proof}

Another application of our results is to the Gibbs $su$-states of a partially hyperbolic diffeomorphism.
Let $f$ be partially hyperbolic.  Recall that a {\em Gibbs $u$-state (resp. $s$-state)} is an $f$-invariant probability measure $\mu$ such that the disintegration of $\mu$ along leaves of the $\cW^u$ (resp. $\cW^s$) foliation is absolutely continuous with respect to volume on $\cW^u$ (resp. $\cW^s$) leaves.  A {\em Gibbs $su$-state} is  an $f$-invariant probability measure that is both a Gibbs $u$-state  and $s$-state.

\begin{nThm*}{{\bf Corollary E.}} Let $f$ be partially hyperbolic, and let $\cU$ be a neighborhood of $f$ in $\Diff^1(M)$.  Then there exists   a non-empty open set $\cO\subset \cU$ such that for every $g\in \cO$, if $\mu$ is a  Gibbs $su$-state for $f$, then $\mu$ has full support, i.e., $\mathrm{supp}(\mu) = M$.
\end{nThm*}

\begin{proof}
The Corollary C implies that there exists a non-empty open set $\cO\subset \cU$ such that every $g\in \cO$ is accessible.
Continuity of the foliations $\cW^u$ and $\cW^s$ implies that the support of a Gibbs $su$-state is bisaturated. 
\end{proof}

The proof of Theorem B follows the lines of the proof of the main result in \cite{DW}, with necessary modifications in the absence of a global partially hyperbolic structure.

\subsection*{Discussion about stable ergodicity}

For $k\leq r$, we say that a diffeomorphism $f$ is {\em $C^k$-stably ergodic in $\Symp^r(M,\omega)$} if
any diffeomorphism $g$ that is $C^k$-close to $f$ in $\Symp^r(M,\omega)$ is ergodic.
Since we do not know examples of stably ergodic diffeomorphisms in $\Symp^1(M,\omega)$,
a higher smoothness is usually required.
In~\cite{ACW2} we have proved that $C^1$-stable ergodicity is $C^1$-dense among partially hyperbolic diffeomorphisms in $\Diff^r_v(M)$ for any $r>1$
and an important step was Theorem~\ref{t=ACW} above.
One can ask if the same result holds in $\Symp^r(M,\omega)$:

\begin{question}
Is $C^1$-stable ergodicity $C^1$-dense among $\Symp^r(M,\omega)$, $r > 1$?
\end{question}

Again the strategy for addressing this question should be completely different from the volume-preserving case
since the ``non-uniform Anosov property" does not exist for symplectomorphisms.

\section{Notation and  outline of the proof of Theorem B}
Throughout, $M$ denotes a closed Riemannian manifold and $m$ denotes a smooth volume on $M$, normalized so that $m(M)=1$. When $M = M^{2n}$ is equipped with a symplectic structure $\omega$, we  will indicate  so. 
\subsection{Charts}
We introduce for each point $p\in M$ a chart
$$\varphi_p\colon B(0,1)\subset T_pM\to M,\quad \text{with } \varphi_p(0)=p$$
that has the following properties:
\begin{enumerate}
\item The map $p\mapsto \varphi_p$ is ``piecewise continuous in the $C ^1$ topology."\footnote{In \cite{DW} similar charts are constructed, but it is claimed that these can be defined on a fixed domain in $\RR^d$, depending continuously on $p\in M$. This continuity is not possible and also unnecessary.}
More precisely, there exist open sets $U_1,\dots U_\ell \subset M$ and:
\begin{itemize}
\item[--] compact sets $K_1,\dots,K_\ell$ covering $M$ with $K_i\subset U_i$,
\item[--] trivializations $\psi_i\colon U_i\times \mathbb{R}^d\to T_{U_i}M$
such that $\psi_i(\{p\}\times B(0,2))$ contains the unit ball in $T_pM$ for each $p\in U_i$, and
\item[--] smooth maps $\Phi_i\colon U_i\times B(0,2)\to M$,
\end{itemize}
such that each $p\in M$ belongs to some $K_i$, with
$$\varphi_p=\Phi_i\circ \psi_i^{-1} \text{ on } B(0,1)\subset T_pM.$$

\item When a volume or symplectic form has been fixed on $M$, it pulls back under $\varphi_p$
to a constant form on $T_pM$.
\end{enumerate}
\begin{remark}\label{r.transverse}
Given a compact set $X$ with a continuous splitting
$T_XM=E\oplus F$ and $\gamma>0$, one can choose a Riemannian metric on $M$
and the charts $\varphi_p$ such that for each $p\in \Lambda$ and $z\in \varphi_p^{-1}(X)\cap B(0,1)$,
the norm of the orthogonal projection of $D\varphi_p(z)^{-1}(F_{\varphi_p(z)})$ onto $D\varphi_p(z)^{-1}(E_{\varphi_p(z)})$
is less than $\gamma$.
\end{remark}
\medskip

The construction can be done as follows.
Assume that a Riemannian metric on $M$ has been fixed.
We first choose a cover of $TM$ by trivializations
$\psi_i\colon U_i\times \mathbb{R}^d\to T_{U_i}M$
such that $\psi_i(p,0)=0\in T_pM$. The maps $u\mapsto \psi(p,u)$ can be chosen close to  isometries.

Denoting by $B^d(0,R)$ the standard open ball in $\mathbb{R}^d$ with radius $R$,
we then construct finitely many charts
$$\Phi\colon B^d(0,3)\to M$$
such that (after replacing the $U_i$ by smaller open sets, if necessary)
for any $U_i$, there exists one such chart $\Phi$ satisfying $U_i\subset \Phi(B(0,1))$.
We can thus define
$$\Phi_i(p,u):=\Phi(\Phi^{-1}(p)+u) \text{ on }  U_i\times B(0,2),$$
and for each $p\in M$, choose $K_i$ containing $p$ and set $\varphi_p=\Phi_i\circ \psi_i^{-1}$ on $B(0,1)\subset T_pM$.

When $M$ is equipped with a volume form, one can require (by Moser's theorem~\cite{Moser})
that $\varphi$ sends divergence-free vector fields to divergence-free vector fields.
When $M$ is equipped with a symplectic form $\omega$
one can require (by Darboux's theorem) that $\varphi_*\omega$ coincides (up to multiplication by a constant)
with the standard symplectic form $\Sigma dp_i\wedge dq_i$ of $\mathbb{R}^d=\mathbb{R}^{2n}$.
This concludes the construction of the charts.

Given a compact set $K$ with a continuous splitting $T_KM=E\oplus F$, one can first choose
a Riemannian metric such that the norm of the orthogonal projection from $F$ to $E$ is arbitrarily small.
One then chooses the charts $\Phi$ in such a way that the bundles $E$ and $F$ lifted in $B^d(0,3)$
are close to constant bundles (this is possible by the continuity of $E$ and $F$).
Since the $\psi_i$ are close to isometries, this shows that for $p\in \Lambda$,
the bundles $E$ and $F$ lifted by $\varphi_p$ in $B(0,1)\subset T_pM$
are close to $E_p$ and $F_p$ respectively, which are close to orthogonal.
This gives the property stated in Remark~\ref{r.transverse}.

\subsection{Conefields}
A \emph{$k$-conefield} $\cC$ over a subset $U\subset M$ is
a subset of the tangent bundle $T_{U}M$ satisfying:
\begin{itemize}
\item[--] $tv\in \cC$ for any $v\in \cC$ and $t\in \RR$; and
\item[--] there is a continuous subbundle $E\subset T_UM$ with $k$-dimensional fibers
such that $\{v\in \cC: \|v\|=1\}$ is a neighborhood of $\{v\in E: \|v\|=1\}$.
\end{itemize}
We denote $\cC(x):=\cC\cap T_xM$.
The conefield $\cC$ is \emph{invariant} under a diffeomorphism $f$ if for any $x\in U\cap f^{-1}(U)$,
the image of $\overline{\cC(x)}$ is contained in ${\rm Interior}(\cC(f(x)))\cup \{0\}$.

A conefield $\cC'$ over $U$ is a {\em $\delta$-perturbation of $\cC$ with support in $V\subset U$} if there exists
a diffeomorphism $h$ that is $\delta$-close to the identity in the $C^1$ topology
such that $h(U)=U$, $h$ coincides with the identity on $U\setminus V$ and
$h^\ast(\cC)=\cC'$.
A $k$-conefield $\mathcal{C}$ is {\em $\delta$-close} to a subbundle $E$ of $T_UM$ with $k$-dimensional fibers
if $\{v\in \cC: \|v\|=1\}$ is $\delta$-close to $\{v\in E: \|v\|=1\}$ in the Hausdorff distance.
\medskip

Let $f\colon M\to M$ be a diffeomorphism, and let $\Lambda$ be a compact $f$-invariant set with a partially hyperbolic splitting $T_\Lambda M = E^u\oplus E^c\oplus E^s$.  

A neighborhood $U$ of $\Lambda$ is {\em admissible} if there exist continuous conefields $\C^u, \C^s, \C^{cu}, \C^{cs}$  over $U$ containing  $E^u, E^s, E^{cu}, E^{cs}$ on $T_\Lambda M$ with the appropriate invariance, transversality and contraction properties. The following proposition is standard.  

\begin{proposition}\label{p=conefields} For every partially hyperbolic set $\Lambda$ for $f$, there exist neighborhoods $\cU_0$ and $U_0$ of $\Lambda$ and $f$,
and conefields  $\C^u_0, \C^s_0, \C^{cu}_0, \C^{cs}_0$ on $U_0$ with the following property.
If $\Delta\subset U_0$  is a  compact $g$-invariant set for $g\in \cU_0$, then it is partially hyperbolic and $U_0$ is an admissible neighborhood of $\Delta$ with respect the conefields  $\C^u_0, \C^s_0, \C^{cu}_0, \C^{cs}_0$ and the diffeomorphism $g$.
\end{proposition}

\subsection{Bi-saturated partially hyperbolic sets. Accessibility}\label{ss=bisaturated}

Consider a diffeomorphism $f$ and a partially hyperbolic set $\Lambda$.
We denote by $U_0(f,\Lambda)$ and $\cU_0(f,\Lambda)$ the neighborhoods given by Proposition~\ref{p=conefields}.

Let $M$ be a manifold of dimension $d\geq 2$, and let $K$ be a compact subset of $M$.
 A {\em $k-$dimensional topological lamination}
 $\cL$ of $K$ is a decomposition of $K$ into path-connected subsets
\[K = \bigcup_{x\in K} \cL(x)
\]
called {\em leaves}, where $x\in \cL(x)$, and two  leaves $\cL(x)$ and $\cL(y)$ are either disjoint or equal,  and a covering of  $K$ by coordinate neighborhoods $\{U_\alpha\}$ with local coordinates $(x^1_\alpha,
\dots, x^d_\alpha)$ with the following property.  For $x\in U_\alpha\cap K$, denote by $\cL_{U_\alpha}(x)$ the connected component of $\cL(x)\cap U_\alpha$ containing $x$. Then in coordinates on $U_\alpha$ the local leaf $\cL_{U_\alpha}(x)$ is  given by a set of equations of the form $x^{k+1}_\alpha=\dots=x^d_\alpha= cst$. If the local coordinates $(x_{\alpha}^1,\dots, x_\alpha^d)$ can be chosen uniformly $C^r$ along the local leaves (i.e., to have uniformly $C^r$ overlaps on the sets  $x^{k+1}_\alpha=\dots=x^d_\alpha= cst$) then we say that $\cL$ {\em has $C^r$ leaves}. 

Note that the leaves of a lamination with $C^r$ leaves are $C^r$, injectively immersed submanifolds of $M$.   A lamination of $M$ is called a foliation.

Let $\Lambda$ be a bisaturated partially hyperbolic set for $f$.
 Continuity and invariance of the partially hyperbolic splitting implies that $\Lambda$ is $u-$ (resp., $s-$) saturated if and only if  $\{\cW^u_f(x) : x\in \Lambda\}$ (resp.,  $\{\cW^u_f(x) : x\in \Lambda\}$) is a lamination of $\Lambda$.

\begin{lemma}\label{l=bisaturatedboundary} Let $\Lambda$ be a bisaturated partially hyperbolic set for $f$.  Then the boundary $\partial \Lambda$ is also bisaturated.
\end{lemma}
\begin{proof} A  set is bisaturated if and only if it is  laminated  by $\cW^s$ leaves and by $\W^u$ leaves.
If $x,y$ belong to the same leaf of a compact lamination $\mathcal W\subset M$, then there exist neighborhoods $V_x$
and $V_y$ in $\mathcal W$, of $x$ and $y$ respectively, that are homeomorphic; thus $x$ belongs to the interior of $\mathcal W$
if and only if $y$ does.
\end{proof}

Let $\P(M)$ be the collection of all subsets of $M$.
We say that $(f,\Lambda)$ is {\em accessible on $X\in \P(M)$} if 
for every $p \in X\cap \Lambda$, and every $q\in X$, there is an $su$-path  for
$(f,\Lambda)$  from $p$ to $q$.   In particular, if $X\cap \Lambda\neq \emptyset$, and $f$ is accessible on $X$, then $\Lambda\supset X$.

We say $(f,\Lambda)$ is \emph{stably accessible} on $X\in \P(M)$ if there exists a neighborhood $U\subset U_0(f,\Lambda)$ of $\Lambda$ with $X\subset U$, and a neighborhood $\cU\subset \cU_0(f,\Lambda)$ of $f$ such that for every  $\widetilde f\in \cU$ and every $\widetilde f$-invariant bisaturated compact set $\widetilde \Lambda\subset U$,
we have that $(\widetilde f,\widetilde \Lambda)$ is  accessible on $X$.  

We say that a set  $X \in \P(M)$  is a {\em $c$-section for $(f,\Lambda)$}  if for every bisaturated subset $\Delta\subset \Lambda$, we have $X\cap\Delta\neq \emptyset$.

\subsection{Admissible families of disks}
Since we do not assume that $E^c$ is tangent to a foliation,
we will work with approximate center manifolds.

For $\rho>0$ small and $p\in \Lambda$, we denote by $B^c(0,\rho)$ the ball inside $E^c_p$ of radius $\rho$ 
and  set
\[V_\rho(p) := \varphi_p(B^c(0,\rho)).\]
We refer to $V_\rho(p)$
as a {\em $c$-admissible disk} (with respect to $(f,\Lambda)$)
with center $p$ and radius $\rho$
and write $r(V_\rho(p))=\rho$.
If $D$ is a $c$-admissible disk with center $p$ and radius
$\rho$, then for $\beta\in(0,1)$,
we denote by $\beta D$ the $c$-admissible
disk with center $p$ and  radius $\beta\rho$.
A {\em $c$-admissible family} (with respect to $(f,\Lambda)$) is a 
finite collection  of pairwise disjoint,
$c$-admissible disks.
\medskip

Define the return time
$R: \P(M) \to {\bf N} \cup\{\infty\}$
as follows. 
For $X\in \P(M)$, let  $R(X)$ be the smallest $J\in {\bf N}\cup\{\infty\}$ satisfying:
\begin{eqnarray}\label{e=Ndef}
f^i(X)\cap X\neq\emptyset, &\hbox{with}& |i|= J+1.
\end{eqnarray}
Note that $R(B_\rho(p))\to \hbox{per}(p)$, as $\rho\to 0$, where
we set $\hbox{per}(p) = \infty$ if $p$ is not periodic.

For ${\mathcal D}$ a  $c$-admissible family and
$\beta\in(0,1)$, we introduce the following notation:
\begin{eqnarray*}
\beta{\mathcal D} &=& \{\beta\Z\,\mid\, \Z\in{\mathcal D}\},\\
|{\mathcal D}| &=& \bigcup_{\Z\in{\mathcal D}} \Z,\\
r({\mathcal D}) &=& \sup_{\Z\in{\mathcal D}} r(\Z),\hbox{ and}\\
R({\mathcal D}) &=& R(|{\mathcal D}|).
\end{eqnarray*}

\subsection{Two propositions}

Our first proposition is the counterpart to  \cite[Lemma 1.1]{DW}.\footnote{While Lemma 1.1 is stated correctly in \cite{DW}, its proof has an error.  In particular, in Lemma 3.3, $r(\cD)$ and $R(\cD)$ are chosen after $\theta$ is given, when they should be chosen before.  A correct proof is given here.}

\begin{proposition}[Stable accessibility on center disks] \label{l=localaccess}
Let $\Lambda$ be partially hyperbolic for $f$, and let $\delta>0$ be given.
Then there exist $J\geq 1$ and a neighborhood $U$ of $\Lambda$ satisfying $\overline U\subset U_0(f,\Lambda)$ and the following property.

If ${\mathcal D}$ is a $c$-admissible family with respect to $(f,\Lambda)$
with  $r({\mathcal D})< J^{-1}$ and
$R({\mathcal D})>J$,  then for all 
$\sigma>0$  there exists $g\in \cU_0(f,\Lambda)$ such that:

\begin{enumerate}
\item $d_{C^1}(f,g)<\delta$,
\item $d_{C^0}(f,g)<\sigma$,
\item for each $D \in {\mathcal D}$, and every bisaturated partially hyperbolic set $\Delta\subset U$ for $g$, we have that
$(g,\Delta)$ is stably accessible on $D$.
\end{enumerate}
\end{proposition}

The second proposition is the counterpart to   \cite[Lemmas 1.2 and 1.3]{DW}.

\begin{proposition}[Stable $c$-sections exist]  \label{l=stablecsections}   Let $\Lambda$ be a  partially hyperbolic set for $f$. Then there exists $\delta>0$ with the following property.

Let $U$ be a neighborhood of $\Lambda$ satisfying $\overline U\subset U_0(f,\Lambda)$.  
For any $J\geq 1$ there exists a $c$-admissible disk family $\cD$  and $\sigma>0$ such that:
\begin{enumerate}
\item $r(\cD) < J^{-1}$,
\item $R({\mathcal D})>J$, and
\item if $g$ satisfies  $d_{C^1}(f,g)<\delta$ and $d_{C^0}(f,g)<\sigma$,
then for any bisaturated partially hyperbolic set $\Delta\subset U$ for $g$, the set $|\cD|$ is a $c$-section for $(g,\Delta)$.
\end{enumerate}
\end{proposition}

\subsection{Proof of Theorem B from Propositions~\ref{l=localaccess} and \ref{l=stablecsections}}
Let $f$, $\Lambda$ and $\cU$  be given as in the statement of the theorem.   Let $\delta>0$ be given by Proposition~\ref{l=stablecsections}.  By shrinking the value of $\delta$ if necessary, we may assume that $d_{C^1}(f,g) < \delta$ implies $g\in \cU$.    

Let the neighborhood $U$  of $\Lambda$ and $J\geq 1$ be given by  Proposition~\ref{l=localaccess}, using the value $\delta/2$. We may assume that $\overline U\subset U_0(f,\Lambda)$.   Let $\cD$ and $\sigma$ be given by Proposition~\ref{l=stablecsections}.
Applying Proposition~\ref{l=localaccess} to $f$, $\Lambda$, $\delta/2$, $\cD$, $\sigma/2$ we associate a perturbation $g_0$ of $f$ satisfying:

\begin{enumerate}
\item $d_{C^1}(f,g_0)<\delta/2$,
\item $d_{C^0}(f,g_0)<\sigma/2$,
\item for each $D \in {\mathcal D}$, and every bisaturated partially hyperbolic set $\Delta\subset U$ for $g_0$, we have that
$(g_0,\Delta)$ is stably accessible on $D$.
\end{enumerate}
By compactness of the Hausdorff topology,
there exists a neighborhood $\cO\subset \cU$ of $g_0$
in the $\delta/2$-neighborhood of $g_0$ such that accessibility holds on each $D$,  for bisaturated sets of any $g\in \cO$.  Then for any $g\in \cO$, we have
$d_{C^1}(f,g)<\delta$ and $d_{C^0}(f,g)<\sigma$.  Let $\Delta\subset U$ be a bisaturated set for such a $g$.

On the one hand, Proposition~\ref{l=stablecsections} implies that $|\cD|$ is a $c$-section for $\Delta$, and so there exists $D\in\cD$ such that $\Delta\cap D\neq \emptyset$.  On the other hand, Proposition~\ref{l=localaccess} then implies that $\Delta\supset D$.  By saturating $D$ by local stable and unstable manifolds for $\Delta$ and using again the bisaturation of $\Delta$, we see that $\Delta$ has nonempty interior.

Consider any point $p\in \Delta$ and its accessibility class $C(p)$, i.e. the set of points $p'\in \Delta$
that can be connected to $p$ by a su-path in $\Delta$. Note that the closure of $C(p)$ is a bisaturated set and hence meets the $c$-section $|\cD|$ at a point $z$.
This point belongs to a disc $\varphi_{x_i}(B^c(0,\rho_i))\subset |\cD|$.
Any point $y$ close to $z$ can be joint by a su-path with two legs to a point in $\varphi_{x_i}(B^c(0,\rho_i))$:
this proves that $C(p)$ intersects $|\cD\|$ at a point $z_p$.
If $q$ is another point in $\Delta$, its accessibility class $C(q)$ meets $|\cD\|$ as well at a point $z_q$ and the stable accessibility relative to $|\cD\|$
implies that the two points $z_p,z_q$ can be connected by a su-path.
We have thus proved that $p$ and $q$ belong to the same accessibility class, hence that $\Lambda$ is accessible,
completing the proof of Theorem B. \endproof

\section{Proof of Proposition~\ref{l=localaccess}}

Fix $f,\Lambda,\delta$ as in the statement of Proposition~\ref{l=localaccess}.
We denote by $c$ the dimension of the center bundle.
The proof follows closely the proof of Lemma 1.1 in~\cite{DW}.
The main adaptation is that we work inside bisaturated sets $\Delta$ for $g$
in a small neighborhood of $\Lambda$ and consider unstable and stable holonomies in restriction to $\Delta$.
\medskip

The partially hyperbolic splitting for $f$ at a point $z$ will be denoted by
$E^u_z\oplus E^c_z\oplus E^s_z$, whereas the splitting for another diffeomorphism $g$
will be denoted by $E^u_{g,z}\oplus E^c_{g,z}\oplus E^s_{g,z}$.
As before $d=\dim(M)$.

For any $p\in M$ we have defined a chart $\varphi_p\colon B(0,1)\subset T_pM\to M$.
From Remark~\ref{r.transverse}, we can assume that
for any $p\in \Lambda$ and $z\in B(0,1)\cap \varphi_p^{-1}(\Lambda)$,
the orthogonal projections of $E^s_z$ on $E^u_z\oplus E^c_z$ and of $E^u_z$ on $E^c_z\oplus E^s_z$
have norms smaller than $10^{-1}$.

In the following, we will reduce the $C^1$-size $\delta$ of the perturbation, the size of the neighborhood $U$ of $\Lambda$,
and the size $\rho$ of the $c$-admissible discs.

\subsection{A center covering.}
We will need to replace $c$-admissible discs by families of disjoint smaller balls.

\begin{lemma}\label{c.cover0}
There exist $\delta_1,\rho_1>0$, $K>1$ and a neighborhood $U_1$ of $\Lambda$ such that
for any $\rho\in (0,\rho_1)$, any $c$-admissible disc $D$ with radius $\rho$, centered $p\in \Lambda$, and for any $\varepsilon\in (0,K^{-1}\rho)$, there exist $z_1,\dots, z_\ell\in T_pM$ such that:
\begin{itemize}
\item[(1)] The balls $B(z_i,100d^2\varepsilon)$ are in the $K\varepsilon$-neighborhood of $\varphi_p^{-1}(D)$.
\item[(2)] The balls $B(z_i,100d^2\varepsilon)$ are pairwise disjoint.
\item[(3)]For any $x\in D$, there exists $z_i$
such that for any $g$ that is $\delta_1$-close to $f$ in the $C^1$ distance and for any bisaturated set $\Delta\subset U_1$ for $g$:
\begin{itemize}
\item[(a)] if $x\in \Delta$ then there is a su-path for $g$ between $x$ and $\varphi_{p}(B(z_i,\varepsilon))$,
\item[(b)] if $\varphi_{p}(B(z_i,\varepsilon))\subset \Delta$, then any point $y\in \varphi_p(B(x,\varepsilon/2))$
belongs to an su-path that intersects $\varphi_{p}(B(z_i,\varepsilon))$.
\end{itemize}
\end{itemize}
\end{lemma}
\begin{proof}
There exists $K_0>1$ such that for any $\varepsilon>0$,
the unit ball $B(0,1)\subset \RR^c$ can be covered by balls
$B(x_1,\varepsilon/4),\dots,B(x_\ell,\varepsilon/4)$ with the property that
any ball $B(x_i,200d^2\varepsilon)$ intersects at most $K_0-1$ others.

We introduce a local flow along the unstable leaves $W^u_g$ of $\Delta$ for $g$.
Fix  $p\in \Lambda$, and for $x\in B(0,\rho_1)\cap \varphi^{-1}_p(\Delta)\subset T_pM$, denote by $\pi^u_x$ the projection along $(E^u_p)^\perp$ of  $B(0,2\rho_1)$ onto the connected component of
$\varphi_p^{-1}(W^u_{g,\varphi_p(x)})\cap B(0,2\rho_1)$ containing $x$.

For each tangent vector $v^u\in E^u_p$, we define a vector field $X_{v^u}$ along the local leaves of $\varphi_p^{-1}(W^u_{g})$ as follows:
for $x\in B(0,\rho_1)\cap \varphi^{-1}_p(\Delta)\subset T_pM$,
 let 
\[X_{v^u}(x) = D\pi^u_x(x+v^u).\]
The vector field $X_{v^u}$ induces a local flow $\Phi^u$ on the set $B(0,2\rho_1)\cap \varphi^{-1}_p(\Delta)$, for   $|t|<\rho_1$:
the orbit of $x$ is the projection by $\pi^u_x$ of the curve $t\mapsto x+tv^u$.
The orbits are $C^1$ curves with a tangent space arbitrarily close to $\RR v^u$ if $\rho_1$,
$\delta_1$ and $U_1$ have been chosen small enough.

Let $D$ be a $c$-admissible disk centered at $p$, with radius $\rho<\rho_1$.
From the property above, one can choose points $x_1,\dots,x_\ell\in E^c_p$
such that the balls $B(x_1,\varepsilon/4),\dots,B(x_\ell,\varepsilon/4)$
cover $\varphi_p^{-1}(D)\subset E^c_p$ and choose integers $k_1,\dots,k_\ell$ in $\{1,\dots,K_0\}$
such that the balls
$B(z_1,100d^2\varepsilon)$, \dots, $B(z_\ell,100d^2\varepsilon)$,
centered at points $z_i:=x_i+500d^2{k_i}\varepsilon v^u$ are pairwise disjoint.
The two first items are satisfied with $K=1000d^2K_0$.

Since the flow lines under $\Phi^u$ are $C^1$-close to lines parallel to $v^u$,
for each point $x\in D\cap \Delta$, there exists $|t|<10K_0$ such that $\Phi_t(\varphi^{-1}_p(x))$
belongs to one of the balls $B(z_i,\varepsilon/2)$. Hence the unstable manifold of $x$ intersects
$\varphi_p(B(z_i,\varepsilon/2))$.

Conversely if $\varphi_p(B(z_i,\varepsilon))\subset \Delta$,
then the continuous map $H\colon (t,y)\mapsto \Phi^u_t(y)$ defined on
$[-K\varepsilon,K\varepsilon]\times B(z_i,\varepsilon)$ is $\varepsilon/2$-close
in the $C^0$ metric to the map
$(t,y)\mapsto y+tv^u$. Hence $B(z_i),\varepsilon/2)+ [-(K-1)\varepsilon,(K-1)\varepsilon]v^u$
is contained in the image of $H$. By construction $B(x,\varepsilon/2)$ is contained in this image.
We have thus proved that any point in $\varphi_p(B(x,\varepsilon/2))$
belongs to the unstable manifold of some point in $\varphi_p(B(z_i,\varepsilon))$.

The lemma is proved.
\end{proof}

\subsection{A center accessibility criterion.}\label{ss.theta}
Let $\theta>0$, $p\in\Lambda$ and $z\in T_pM$.

We say that the pair $(g,\Delta)$ is \emph{$\theta$-accessible} on the ball
$\varphi_p(B(z,2d\varepsilon))$ if there exist an orthonormal basis $w_1,\dots,w_c$ of $E^c_p$
and for each $j\in \{1,\dots,c\}$, a continuous map
$$H^j\colon [-1,1]\times [0,1]\times \varphi_p^{-1}(\Delta)\cap B(z,2d\varepsilon)\to \varphi_p^{-1}(\Delta)\cap B(0,2\rho)$$
such that for any $x\in \varphi_p^{-1}(\Delta)\cap B(z,2d\varepsilon)$ and $s\in [-1,1]$,
\begin{itemize}
\item[(a)] $H^j(s,0,x)=x$,
\item[(b)] the map $\varphi_p\circ H^j(s,.,x)\colon [0,1]\to \Delta$ is a  $4$-legged su-path, i.e. the concatenation of $4$ curves, each contained in a stable or unstable leaf,
%\item[(c)] the image of the map $\varphi_p^{-1}\circ H^j$ has diameter $<10d\varepsilon$,
\item[(c)] $\|H^j(s,1,x)-x\|< \frac \varepsilon {10d}$, and
\item[(d)] $\|H^j(\pm1,1,x)-(x\pm\theta \varepsilon w_j)\|< \theta \frac \varepsilon {10d}.$
\end{itemize}
\smallskip

The following replaces Lemma 3.2 in~\cite{DW} in our setting.\footnote{The proof of Lemma 3.2 in \cite{DW} contains a subtle error, in the sentence beginning: ``By a standard argument ..." The argument described there is indeed standard in the dynamically coherent setting, but it is not clear in the more general setting.  We bypass this argument here by removing the requirement that the $su$-paths end in the $c$-admissible disk, establishing local accessibility directly.}
\begin{lemma}\label{l.criterion}
For any $\theta>0$, there exist $\delta_2,\rho_2>0$ and a neighborhood $U_2$ of $\Lambda$
such that
\begin{itemize}
\item[--] for any $p\in \Lambda$, any $z$ in the ball $B(0,\rho_2)\subset T_pM$ and $\varepsilon\in (0,\rho_2)$,
\item[--] for any diffeomorphism $g$ which is $\delta_2$-close to $f$ in the $C^1$ topology,
\item[--] for any bi-saturated set $\Delta\subset U_2$ such that
$(g,\Delta)$ is $\theta$-accessible on the ball $\varphi_p(B(z,2d\varepsilon))$,
\end{itemize}
the pair $(g,\Delta)$ is accessible on $\varphi_p(B(z,\varepsilon))$.
\end{lemma}
\begin{proof}
Let $u$, $s$ be the dimensions of the bundles $E^u$, $E^s$. Hence $d=u+c+s$.
Let $v_1,\dots,v_u$ and $v_{u+c+1},\dots,v_{d}$
be orthonormal bases  of $E^u_p$ and $E^s_p$, respectively.
We define local flows $(\Phi^i_t)$ (for $i\in \{1,\dots,u\}\cup\{u+c+1,\dots,d\}$) on $\varphi_p^{-1}(\Delta)$, as in the proof of Lemma~\ref{c.cover0}.

As in the proof of Lemma~\ref{c.cover0}, we define for each $j\in \{1,\dots,u\}$
a local flow $\Phi^j$: at any $x\in B(z,2d\varepsilon)\cap \varphi_p^{-1}(\Delta)$
it is tangent to the vector field obtained by projecting the vector $\theta \varepsilon v^u_j$ at $x$
orthogonally to $E^u$ on the tangent space $\varphi_p^{-1}(W^u(\varphi_p(x)))$.
The point $\Phi^j_t(x)j$ is the projection of $x+t\theta \varepsilon v^u_j$ orthogonally onto $\varphi_p^{-1}(W^u(\varphi_p(x)))$.
Similarly, for $j\in \{1,\dots,s\}$, we define a flow $\Phi^{u+c+j}_t(x)$ in the direction
of $v^s_j$, along the stable leaves of $\varphi_p^{-1}(\Delta)$.
Choosing $\rho_2$, $\delta_2$, $U_2$ small, the tangent spaces of the unstable and stable leaves
of $g$ inside  $\varphi_p^{-1}(\Delta)$ in $B(p,\rho_2)$ are close to $E^u_p$ and $E^s_p$.
This gives
\begin{equation}\label{e.lip}
\|\Phi^i_t(x)-(x+t\theta \varepsilon v_i)\|< |t|\theta \frac \varepsilon {10d}.
\end{equation}

We also define maps in the center direction. Let us set $v_{c+j}=w_j$.
For $j\in\{1\dots,c\}$,
we introduce inductively $\Phi^{c+j}_t(x)$ (while it can be defined) by
$$\Phi^{c+j}_t(x)=H^j(t,1,x) \text{ when } t\in [0,1),$$
$$\Phi^{c+j}_t(x)=\Phi^{c+j}_{t-1}\circ \Phi^{c+j}_1(x) \text{ when } t>1,$$
$$\Phi^{c+j}_t(x)=\Phi^{c+j}_{t+1}\circ \Phi^{c+j}_{-1}(x) \text{ when } t<0,$$
Let us consider a point $x_0\in \varphi_p^{-1}(\Delta)\cap B(z,\varepsilon)$.
From (c), (d) and~\eqref{e.lip}, one can define for each $(t_1,\dots,t_d)\in [-3\theta^{-1},3\theta^{-1}]^d$
$$P(t_1,\dots,t_d)=\Phi^1_{t_1}\dots \Phi^d_{t_d}(x_0).$$
This induces a continuous map satisfying
$$\|P(t_1,\dots,t_d)-(x_0+{\textstyle\sum}_i t_i\theta\varepsilon v_i)\|<\frac {2\varepsilon} {10}.$$
The Brouwer fixed point theorem implies that the image of $P$ contains the ball centered at $x_0$
of radius $(3-\textstyle \frac 1 2)\varepsilon$, and  hence the ball $B(z,\varepsilon)$.
By construction, the points in the image of $\varphi_p\circ P$ are be connected to $z_0$ by an su-path in $\Delta$.
We have thus shown that $(g,\Delta)$ is accessible on $B(z,\varepsilon)$.
\end{proof}

\subsection{Elementary perturbations.}
The perturbation will be built from the
\begin{lemma}\label{l.perturbation}
There exists $\eta,\alpha_0>0$ small with the following properties.

For any $\alpha\in (0,\alpha_0)$, $p\in \Lambda$,
$z\in B(0,1/4)\subset T_pM$,  $ r \in (0,1/4)$ and any unit vector $v\in E^c_p$,
there exists a diffeomorphism $T$ of $T_pM$:
\begin{itemize}
\item[--] which is supported on $B(z,3 r )$,
\item[--] whose restriction to $B(z,2 r )$ coincides with
$$x\mapsto x+\alpha \eta  r  v,$$
\item[--] whose tangent map $DT(y)$ is $\alpha$-close to $\id$ for any $y\in T_{p}M$,
\item[--] which is $\frac r {100d^2}$-close to the identity in the $C^0$ distance.
\end{itemize}
Moreover, if $f$ preserves a volume $m$ or a symplectic form $\omega$, then such a  $T$ can be constructed so that the maps $\varphi_p\circ T\circ \varphi_p^{-1}$ preserve $m$ or $\omega$ as well.
\end{lemma}
\begin{proof}
The construction is standard.
One first notices that it can be done in the case $ r =1/4$.
One then reduces $ r $ by conjugating by an homothety.

With a bump function, one builds a vector fields which takes the constant value $v$ on $B(z,2 r )$
and which vanishes outside $B(z,3 r )$.
There exists $\eta>0$ such that the time $t$ of the flow is at distance $\eta^{-1}.t$ from the identity in the $C^1$-topology.
For $\alpha>0$ small, the map $T$ is the time $\alpha\eta$ of the flow.

In the volume-preserving case, the lift of the volume form is constant in the domain of the charts.
Choosing a divergence free vector field, the map $T$ preserves the volume.

In the symplectic case, the symplectic form in the chart is constant.
The constant vector field is hamiltonian. Using a bump function, one can extend the hamiltonian
to a function which vanishes outside $B^c(z,3 r )$.
The associated vector field is then symplectic as required.
\end{proof}

The diffeomorphism $g$ will be obtained from $f$ as a composition:
\begin{equation}\label{e.def-g}
g:=\Psi_{L}\circ\dots\circ \Psi_1\circ f,\quad
\text{with} \quad \Psi_{\ell}:=\varphi_{p_\ell}\circ T_\ell\circ \varphi_{p_\ell}^{-1},
\end{equation}
where the points $p_\ell$ belong to $\Lambda$ and where
the maps $T_\ell$ are diffeomorphisms of $T_{p_\ell}M$ given by Lemma~\ref{l.perturbation}
which coincide with the identity outside some sets $\Omega_\ell$ contained in $B(0,2\rho)\subset T_{p_\ell}M$,
for some $\rho>0$ small which will be chosen later.
The supports $\varphi_{p_\ell}(\Omega_\ell)$ will be chosen pairwise disjoint so that
the maps $\Psi_\ell$ commute.
\medskip

We consider cone fields $\cC^u_0, \cC^s_0$ on ${U_0}$,
respectively invariant by $f$ and $f^{-1}$ as in Proposition~\ref{p=conefields}.
If $\delta_3>0$ is small,
for any $g$ that is $\delta_3$-close to $f$ in the $C^1$ topology,
the same cone fields $\cC^u_0, \cC^s_0$ are still invariant by $g$ and $g^{-1}$.\hspace{-1cm}\mbox{}
\medskip

We may assume without loss of generality that  $\delta< \min(\delta_1,\delta_2,\delta_3)$.
Recall that the charts $\varphi_p$ depend continuously on $p$ in the $C^1$-topology
when $p$ belong to the atoms of a finite compact covering of $M$. Consequently, there exists $\rho_3>0$ and $\alpha\in (0,\alpha_0)$ small
such that if the support $\Omega_\ell$ of each map $T_\ell$ is contained in $B(0,2\rho)\subset T_{p_\ell}M$
for some $\rho\in(0,\rho_3)$ and has a tangent map $DT$ which is $\alpha$-close to the identity, then
the diffeomorphism $g$ is $\delta$-close to $f$ in the $C^1$ distance.

\subsection{Choice of $J$ and $U$.}
For $J\geq 1$,
let us introduce the  iterates $\cC^u=Df^J(\cC^u_0)$ and $\cC^s=Df^{-J}(\cC^s_0)$ on a neighborhood $U$ of $\Lambda$
satisfying
\begin{equation}\label{e.defU}
\overline U\subset U_1\cap U_2\cap \bigcap_{|k|\leq J} f^k(U_0).
\end{equation}
The contraction of the cone fields ensures that $\cC^u$ and $\cC^s$ get arbitrarily close to the bundles
$E^u$ and $E^s$ (defined on the maximal invariant set of $\overline U$) as $J\to +\infty$.
Hence, there exist $J_1\geq 1$ and $\rho_4>0$ such that
if $J\geq J_1$ and if $\rho<\rho_4$, then for any $p\in \Lambda$,
\begin{itemize}
\item[--] $\varphi_p(B(0,2\rho))\subset U$,
\item[--] the cone fields $D\varphi_p^{-1}(\cC^s)$ and $D\varphi_p^{-1}(\cC^u)$ on $B(0,2\rho)$ are $\gamma$-close to the spaces $E^s_p$ and $E^u_p$
in $T_pM$ for some $\gamma>0$ much smaller than $\alpha\eta$.
\end{itemize}
In particular, from the choice of the Riemannian metric,
for any point $z\in B(0,2\rho)\subset T_pM$, the orthogonal projection of any unit vector $u\in D\varphi_p(z)^{-1}\left(\cC^s(\varphi_p(z)\right)$
on $E^c_p$ has norm smaller than $1/2$.
\medskip

We now fix:
\begin{itemize}
\item[--] $\rho>0$ smaller than $\min(\rho_1,\rho_2,\rho_3,\rho_4)$,
\item[--] $J\geq J_1$ large enough so that any $c$-admissible disk $D$ with center $p\in \Lambda$ and radius $r(D)<J^{-1}$
lifts by $\varphi_p$ as a subset of $B(0,\rho)\subset T_pM$,
\item[--] the neighborhood $U$ to satisfy~\eqref{e.defU},
\item[--] a $c$-admissible family $\cD$ of disks as in the statement of Proposition~\ref{l=localaccess},
\item[--] $\sigma>0$ as in the statement of Proposition~\ref{l=localaccess}.
\end{itemize}
The construction also depends on a number $\varepsilon>0$ smaller than $\sigma$, $\rho_2$ and $K^{-1}r(D)$ for any $D\in \cD$ (as in Lemma~\ref{c.cover0}).
We will specify later the value of $\varepsilon$.

\subsection{Construction of the diffeomorphism $g$.}
We associate to each $c$-admissible disc $D\in \cD$
a set of balls as given by Lemma~\ref{c.cover0}.
The union of these sets defines a family $\cB$ of balls
$B_i:=\varphi_{p_i}(B(z_i,100d^2\varepsilon))$
inside the tangent spaces of points $p_i\in \Lambda$.
Since the discs $D\in \cD$ are disjoint,
by choosing $\varepsilon>0$ small enough the items (a) and (b) in Lemma~\ref{c.cover0}
ensure that the balls $B_i$ are pairwise disjoint.
\medskip

We now define $g\circ f^{-1}$ in each $B_i$ separately.
The choice of $\rho$ gives $\varphi_{p_i}(B(z_i,100d^2\varepsilon))\in U$, and one can choose two spaces $\mathcal{E}^u,\mathcal{E}^s\subset T_{p_i}M$
with the same dimension as $E^u_{p_i}$ and $E^s_{p_i}$ and satisfying
$$D\varphi_{p_i}(z_i).\mathcal{E}^s\subset \mathcal{C}^s(\varphi_{p_i}(z_i)),\quad
D\varphi_{p_i}(z_i).\mathcal{E}^u\subset \mathcal{C}^u(\varphi_{p_i}(z_i)).$$
We choose two unit vectors $e^s\in \mathcal{E}^s$, $e^u\in \mathcal{E}^u$
and we also fix an orthonormal basis $w_1,\dots, w_c$ of $E^c_{p_i}$.

For each $j\in\{1,\dots,c\}$,
the Lemma~\ref{l.perturbation} provides us with:
\begin{itemize}
\item[--] diffeomorphisms $T_{i,j}$ of $T_{p_i}M$,
whose restriction to $B(z_i+ 10 j d\varepsilon e^s,2d\varepsilon)$ coincides with the translation by $\alpha\eta d\varepsilon w_j$,
\item[--] diffeomorphisms $T_{i,-j}$ of $T_{p_i}M$,
whose restriction to $B(z_i- 10 j d\varepsilon e^s,2d\varepsilon)$ coincides with the translation by $-\alpha\eta d\varepsilon w_j$,
\end{itemize}
Moreover $DT_{i,\pm j}$ is $\alpha$-close to the identity and $T_{i,\pm j}$
coincides with the identity
outside $B(z_i\pm 10 j d\varepsilon e^s,3d\varepsilon)$.
\medskip

Since the norm of the orthogonal projection of $\mathcal{E}^s$ to $E^{c}$ is less than $1/2$,
the supports of the $T_{i,j}$ and $T_{i,-j}$ for $j\in\{1,\dots,c\}$, are pairwise disjoint, and also disjoint from
$B(z_i,2d\varepsilon)$.
Also the union of the supports is contained in $B(z_i,100d^2\varepsilon)\subset T_{p_i}M$.
Since the balls $B_i$ are disjoint, the composition of $f$
with all the $\Psi_{i,\pm j}:=\varphi_{p_i}\circ T_{i,\pm j}\circ \varphi_{i,j}^{-1}$
as in~\eqref{e.def-g} defines a diffeomorphism $g$, which is $\delta$-close to $f$
in the $C^1$-topology. Since the diameters of these balls has been chosen small,
$g$ is also $\sigma$-close to $f$ in the $C^0$-topology.

It remains to check the item (3) of Proposition~\ref{l=localaccess}.

\subsection{$\theta$-accessibility.}
We set $\theta={\alpha\eta}d$ and
check the criterion in Section~\ref{ss.theta}.
\begin{lemma}\label{l.theta-access}
If $\varepsilon>0$ is small enough, then
for any diffeomorphism $\widetilde g$ that belongs to a small $C^1$-neighborhood of $g$
and for any bisaturated set $\widetilde \Delta\subset U$ for $\widetilde g$, then
the pair $(\widetilde g,\widetilde \Delta)$ is $\theta$-accessible on each ball $\varphi_{p_i}(B(z_i,2d\varepsilon))$.
\end{lemma}
\begin{proof}
Let us describe the holonomies in each ball $B(z_i,100d^2\varepsilon)\subset T_{p_i}M$.
We introduce the affine foliations $\cF^u_0,\cF^s_0$ of $T_{p_i}M$ by leaves parallel to $\mathcal{E}^u$ and $\mathcal{E}^s$
and the perturbed foliation
$$\cF^u:=(T_{i,1}\circ\dots \circ T_{i,c})(\cF^u_0).$$

We next introduce flows $\Phi^s,\Phi^u$ along the leaves of $\cF^s_0$ and $\cF^u$ in the directions $e^s$ and $e^u$
as in proof of Lemma~\ref{c.cover0}. For the linear foliation $\cF^s$, $\Phi^s$ simply coincides with the
linear flow $(x,t)\mapsto x+te^s$.
For $\cF^u$, one defines $\Phi^u(x,t)$ by projecting the flow $(x,t)\mapsto x+te^u$ on the leaves of $\cF^u$,
along the space $\mathcal{E}^s+E^c_{p_i}$.

\begin{claim}
For each $j\in\{1,\dots,c\}$, the composition
$$\Phi^s_{-10d\varepsilon}\circ \Phi^u_{-10 j d\varepsilon}\circ \Phi^s_{10d\varepsilon}\circ \Phi^u_{10 j d\varepsilon}$$
coincides on $B(z_i,2d\varepsilon)$ with the translation by $\theta\varepsilon w_j$.

Similarly, the composition
$\Phi^s_{10d\varepsilon}\circ \Phi^u_{-10 j d\varepsilon}\circ \Phi^s_{-10d\varepsilon}\circ \Phi^u_{10 j d\varepsilon}$
coincides on $B(z_i,2d\varepsilon)$ with the translation by $-\theta\varepsilon w_j$.
\end{claim}
\begin{proof}
Indeed on $B(z_i,2d\varepsilon)$ the map $\Phi^u_{10 j d\varepsilon}$ coincides with $x\mapsto x+(10jd\varepsilon)e^u$
(by construction the support of the maps $T_{i,j}$
does not intersect $B(z_i,2d\varepsilon)+\mathbb{R} e^u$).
On $B(z_i,2d\varepsilon)+(10jd\varepsilon)e^u+10d\varepsilon$,
the map $\Phi^u_{-10 j d\varepsilon}$ coincides with the composition of $x\mapsto x-(10jd\varepsilon)e^u$
with the translation by $\alpha\eta d\varepsilon w_j=\theta\varepsilon w_j$.
The first part of the claim follows. The second is obtained analogously.
\end{proof}

Arguing in a similar way and using the fact that the $C^0$ size of the perturbations
$T_{i,j}$ is smaller than $\frac{\varepsilon}{100d}$, one gets:

\begin{claim}
For each $j\in\{1,\dots,c\}$ and $s\in [0,1]$ the compositions
$$\Phi^s_{-s10d\varepsilon}\circ \Phi^u_{-s10 j d\varepsilon}\circ \Phi^s_{s10d\varepsilon}\circ \Phi^u_{s10 j d\varepsilon}$$
$$\text{ and } \Phi^s_{s10d\varepsilon}\circ \Phi^u_{-s10 j d\varepsilon}\circ \Phi^s_{-s10d\varepsilon}\circ \Phi^u_{s10 j d\varepsilon}$$
are at a $C^0$-distance smaller than $\frac{\varepsilon}{10d}$ from the identity.
\end{claim}

Let $\widetilde g$ be a diffeomorphism $C^1$-close to $g$
and $\widetilde \Delta\subset U$ be a bisaturated set for $\widetilde g$.
At each $x$ in $B(0,2\rho)\cap \varphi_{p}^{-1}(\Delta)$,
the connected components of the leaves $\varphi_{p}^{-1}(W^{s/u}_{\widetilde g}(\varphi_{p}(x)))\cap B(0,2\rho)$
containing $x$ define leaves $W^{s/u}_{\widetilde g,loc}(x)$.
One can define flows $\widetilde \Phi^s$ and $\widetilde \Phi^u$
in $B(0,2\rho)\cap \varphi_{z_i}^{-1}(\widetilde \Delta)$ as before:
$\Phi^u(x,t)$ is obtained by projecting $(x,t)\mapsto x+te^u$ on the leaves $W^{u}_{\widetilde g,loc}(x)$
along the space $\mathcal{E}^u+E^c_{p_i}$;
$\Phi^s(x,t)$ is obtained by projecting $(x,t)\mapsto x+te^s$ on the leaves $W^{s}_{\widetilde g,loc}(x)$
along the space $\mathcal{E}^s+E^c_{p_i}$.

We then define the map
$$H^j\colon [0,1]\times [-1,1]\times \varphi_p^{-1}(\Delta)\cap B(z_i,2d\varepsilon)\to \varphi_p^{-1}(\Delta)\cap B(0,2\rho).$$
For $s\in [-1,1]$ and $x$ in $\varphi_p^{-1}(\Delta)\cap B(0,2d\varepsilon)$, the arc $t\mapsto H^j(s,t,x)$
is the concatenation of the four arcs
$$t\mapsto \Phi^u_{|s|t40 j d\varepsilon}(x),$$
$$t\mapsto \Phi^s_{st40 j d\varepsilon}(\Phi^u_{|s|40 j d\varepsilon}(x)),$$
$$t\mapsto \Phi^u_{-|s|t40 j d\varepsilon}(\Phi^s_{s40 j d\varepsilon}\circ \Phi^u_{|s|40 j d\varepsilon}(x)),$$
$$t\mapsto \Phi^s_{-st40 j d\varepsilon}(\Phi^u_{-|s|40 j d\varepsilon}\circ \Phi^s_{s40 j d\varepsilon}\circ\Phi^u_{|s|40 j d\varepsilon}(x)).$$
The items (a), (b) of the definition of the $\theta$-accessibility hold by construction.
\medskip

The stable leaves $W^{s}_{\widetilde g,loc}(x)$ are $C^1$-close to the leaves of the foliation
$\cF^s_0$. Indeed the $J$ first iterates of $g$ and $f$ coincide on $B_i$ (since $D(\cD)>J$)
so that the tangent spaces of the unstable leaves of $g$ are tangent to the cone field $\cC^s$
and its preimage by $\varphi_{p_i}$ is
$\gamma$-close to the direction $\mathcal E^s$ on $B(z_i,100d^2\varepsilon)$.
The same holds for $\widetilde g$ which is $C^1$-close to $g$.

Similarly the unstable leaves of $g$ and $\widetilde g$ are tangent to the cone field $\cC^u$
on $f^{-1}(B_i)$. Hence on each ball
$\varphi_{p_i}(B(z_i\pm 10 j d\varepsilon e^s,3d\varepsilon))$, they are
tangent to the cone field $D\Psi_{i,\pm j}(\cC^u)$.
Their preimages by $\varphi_{p_i}$ are thus
$\gamma$-close to the leaves of $\mathcal F^u$ on $B(z_i,100d^2\varepsilon)$.
This implies that the trajectories of the flows $\Phi^{s/u}$ and $\widetilde \Phi^{s/u}$ are $C^1$-close.
Together with the two previous claims, it gives the items (c) and (d) of the definition of the $\theta$-accessibility.
\end{proof}

\subsection{Conclusion of the proof of Proposition~\ref{l=localaccess}.}
Let $\Delta\subset U$ be a bisaturated set for $g$.
Let us consider any diffeomorphism $\widetilde g$ that is $C^1$-close to $g$
and any bisaturated set $\widetilde \Delta\subset U$ for $\widetilde g$
(in particular, any bisaturated $\widetilde \Delta$ for $\widetilde g$ contained in a small neighborhood of
$\Delta$).
The pair $(\widetilde g,\widetilde \Delta)$ is $\theta$-accessible on each ball $\varphi_{p_i}(B(z_i,2d\varepsilon))$ by Lemma~\ref{l.theta-access}.
By Lemma~\ref{l.criterion}, the pair 
$(\widetilde g,\widetilde \Delta)$ is accessible on each ball $\varphi_{p_i}(B(z_i,\varepsilon))$.

Let $D$ be any $c$-admissible disk in the family $\cD$, which intersects $\widetilde \Delta$ at a point $z$.
By Lemma~\ref{c.cover0}(3a), there exists $\varphi_{p_i}(B(z_i,\varepsilon))$ and a su-path for $\widetilde g$
between $z$ and a point $y\in \varphi_{p_i}(B(z_i,\varepsilon))$.
By accessibility, $\varphi_{p_i}(B(z_i,\varepsilon))\subset \widetilde\Delta$. By Lemma~\ref{c.cover0}(3b),
any point $x$ in the $\varepsilon/2$-neighborhood of $z$ in $D$
can be connected by a su-path to a point $y'\in \varphi_{p_i}(B(z_i,\varepsilon))$.
By accessibility, $y,y'\in \varphi_{p_i}(B(z_i,\varepsilon))$ belong to a su-path.
This shows that any point in the $\varepsilon/2$-neighborhood of $z$ in $D$
can be connected to $z$ by a su-path.
Since $D$ is connected, any two points in $D$ belong to
a su-path in $\widetilde \Delta$, showing that $(\widetilde g,\widetilde \Delta)$ is accessible on $D$.
This concludes the proof of the stable accessibility of $(g,\Delta)$ on any disk $D\in \cD$,
and of the last item of Proposition~\ref{l=localaccess}.
\endproof

\section{Proof of Proposition~\ref{l=stablecsections}}
For $\Lambda$  a partially hyperbolic set for $f$ and
$U$ a neighborhood of $\Lambda$ with $\overline U\subset U_0(f,\Lambda)$, we denote
the maximal $f$-invariant set in $\overline U$ by
\[\Lambda_U(f) : = \bigcap_{j\in \ZZ} f^{j}(\overline U).\]
For $J\geq 1$, denote by $\hbox{Per}_{\leq J}(f,U)$ the set of points
\[\hbox{Per}_{\leq J}(f, U) := \{p\in \overline U\,\mid\, \hbox{per}(p) \leq J\}.
\]

\subsection{Existence of c-admissible families}
We restate~\cite[Lemma 2.3]{DW} in our setting.
\begin{lemma}\label{l=csectionexistence}
Let $\Lambda$  be a partially hyperbolic set for $f$, let $\beta\in(0,1)$,
let $U$ be neighborhood of $\Lambda$ with $\overline U\subset U_0(f,\Lambda)$ and let $J\geq 1$.

Then for every $\rho>0$ sufficiently small and denoting by
$B_\rho$ the $\rho$-neighbor\-hood of $\hbox{Per}_{\leq J}(f, U)$,
there exist  $q_1,\ldots , q_k \in \Lambda_U(f)\setminus B_\rho$
such that
\begin{enumerate}
\item the balls $B_{\beta \rho}(q_1)$,\dots, $B_{\beta \rho}(q_k)$
cover $\Lambda_U(f)\setminus B_\rho$,
\item $V_{\rho}(q_i)\cap V_{\rho}(q_j)=\emptyset$, for all $i\neq j$, and
\item $V_{\rho}(q_i)\cap f^m\left(V_{\rho}(q_j)\right)=\emptyset$, for all $i,j$ and $0<|m|\leq J$.
\end{enumerate}
In particular $\cD : =\{ V_\rho(q_1),\ldots,V_\rho(q_k) \}$ is a $c$-admissible disk family with respect to $(f,\Lambda_U(f))$ satisfying
$R({\mathcal D}) > J.$
\end{lemma}
\begin{remark}
there exists $\sigma>0$ such that for any diffeomorphism $g$ with $d_{C^0}(g,f)<\sigma$,
the set $\Lambda_U(g)$ is contained in an arbitrarily small neighborhood of $\Lambda_U(f)$
so that we still have $\Lambda_U(g)\setminus B_\rho\subset \cup_i B_{\beta\rho}(q_i)$.
\end{remark}

\subsection{From $c$-admissible families to stable $c$-sections}
The conclusions of Proposition~\ref{l=stablecsections} follow from the next lemma, choosing
 $\rho<J^{-1}$.

\begin{lemma}\label{l=ccoverscsections}   Given  a  partially hyperbolic set $\Lambda$ for $f$,
there exist $\delta,\beta>0$ with the following property.

Let $U$ be neighborhood of $\Lambda$ with $\overline U\subset U_0(f,\Lambda)$,  let $J\geq 1$ and, given $\rho>0$ sufficiently small,
let  $\cD$ be the c-admissible disk family given by Lemma~\ref{l=csectionexistence}. %for $f,\Lambda,U, J, \rho$, and $\beta/2$. 

Then there exists $\sigma>0$ such that
\begin{itemize}
\item[--] for any diffeomorphism $g$ satisfying $d_{C^1}(f,g)<\delta$ and $d_{C^0}(f,g)<\sigma$,
\item[--] for any bisaturated partially hyperbolic set $\Delta\subset U$ for $g$,
\end{itemize} 
the set $|\cD|$ is a $c$-section for $g$.
\end{lemma}

\begin{proof}
Consider $f,\Lambda,U,J$ satisfying the hypotheses.
The set  $\hbox{Per}_{\leq J}(f,U)$ is compact, partially hyperbolic and satisfies for every  $p\in \hbox{Per}_{\leq J}(f,U)$:
\begin{equation}\label{e=lowperiod}
\left(\cW^s(p)\cup \cW^u(p)\right) \cap \hbox{Per}_{\leq J}(f,U) = \{p\}.
\end{equation}

The following property follows easily from (\ref{e=lowperiod}).
\begin{claim}\label{c1}
\label{l=escape} There exist $\rho_0,\delta_0>0$
such that the set $B_0:= N_{\rho_0}(\hbox{Per}_{\leq J}(f,U))$ has the following properties.
For every diffeomorphism $g$ with $d_{C^1}(f,g)<\delta_0$, every
$p\in B_0\cap \Lambda_U(g)$ can be connected to a point in  $\bar U\setminus B_0$ by an $su$-path for $g$ with one leg.
\end{claim}

There is also  a projection onto admissible discs by $su$-paths.

\begin{claim}\label{c2}  There exist $\beta\in (0,1)$ and $\rho_1,\delta_1> 0$
such that for every diffeomorphism $g$ with $d_{C^1}(f,g)<\delta_1$,
every $\rho \in(0, \rho_1)$, every bisaturated set $\Delta\subset U$ for $g$,
every $q\in \Lambda_U(g)$ and $p\in B_{\beta\rho}(q)\cap \Delta$,
there is a $su$-path  for $g$ with $2$ legs from $p$ to some point in $V_{\rho}(q)$.
\end{claim}
\begin{proof}
The proof follows the same argument as in the proofs of Lemmas~\ref{c.cover0} and~\ref{l.criterion}.
Let $B^u(0,1)$ and $B^s(0,1)$ be the unit balls in the spaces $E^u_q$ and $E^s_q$.
Working in the chart $\phi_q$,
define a continuous map $\Phi$ from $B^u(0,1)\times B^s(0,1)$ to
$\varphi_q^{-1}(\Delta)\subset T_qM$ in the following way.
For $v,w\in B^u(0,1)\times B^s(0,1)$,  first project orthogonally along the space $E^c_q\oplus E^s_q$
the point $\varphi^{-1}_q(p)+\rho v$ onto the unstable leaf of $\varphi^{-1}_q(p)$ lifted in the chart:
this defines a point $y\in \varphi_q^{-1}(\Delta)$.
Then project orthogonally along the space $E^u_q\oplus E^c_q$
the point $y+\rho w$ onto the stable leaf of $y$ in the chart:
this defines the point $\varphi_p^{-1}(\Phi(v,w))$. It belongs to the bisaturated set $\Delta$.

The restriction of $\Phi$ to the boundary of $B^u(0,1)\times B^s(0,1)$ is disjoint from $E^c_q$,
and the Brouwer fixed point theorem ensures that the image of $\varphi_p^{-1}\circ \Phi$ intersects $E^c_q$.
Since the stable and unstable leaves of $g$ lifted in the chart are close to the directions $E^u_q$ and $E^s_q$,
the intersection point belongs to $B^c(0,\rho)\subset E^c_q$.
\end{proof}

Fix $\delta=\min(\delta_0,\delta_1)$, $\rho < \min\{\rho_0,\rho_1\}$ and let $\cD$ be the c-admissible disk family
with centers $q_1,\ldots , q_k \in \Lambda_U(f)\setminus B_\rho$ given by  Lemma~\ref{l=csectionexistence}.
Let $\sigma>0$ be given by the remark following Lemma~\ref{l=csectionexistence}.
Consider a diffeomorphism $g$ satisfying $d_{C^1}(f,g)<\delta$ and $d_{C^0}(f,g) < \sigma$,
and let $\Delta\subset U$ be a bisaturated partially hyperbolic set for $g$.

 Claim~\ref{c1} gives that
every $p\in B_\rho \cap \Delta$ can be connected to a point in 
$\Lambda_U(g)\setminus B_\rho$ by a $su$-path for $g$. Since $\Delta$ is bisaturated, it intersects $\Lambda_U(g)\setminus B_\rho$.

By Lemma~\ref{l=csectionexistence} and the remark that follows,
the balls $B_{\beta\rho}(q_1), \cdots, B_{\beta\rho}(q_k)$ cover $\Lambda_U(g)\setminus B_\rho$.
Hence, there exists $i$ such that $\Delta \cap B_{\beta\rho}(q_i) \neq \emptyset$.  Now the Claim~\ref{c2} and the bisaturation of $\Delta$ imply
that $\Delta\cap V_\rho(q_i)\neq \emptyset$.

We have thus shown that $\Delta$ intersects $|\cD|$,
as required.
\end{proof}

\end{document}